\documentclass{singlecol-new}

\usepackage{natbib,stfloats}
\usepackage{mathrsfs}
\usepackage{eucal,graphicx,float}
\usepackage[center]{caption}
\usepackage{color,soul}

\numberwithin{table}{section}

\newcommand{\A}{{\mathcal A}}
\newcommand{\B}{{\mathcal B}}

\newcommand{\N}{{\mathbb N}}
\newcommand{\R}{{\mathbb R}}
\newcommand{\Z}{{\mathbb Z}}

\newcommand{\cref}[1]{Cor.~\ref{#1}}
\newcommand{\egref}[1]{Example~\ref{#1}}
\newcommand{\fall}{\quad\text{for all }}
\renewcommand{\d}{\,{\mathrm d}}
\DeclareMathOperator{\dist}{dist}			
\DeclareMathOperator{\erf}{erf}				

\theoremstyle{TH}{
\newtheorem{lemma}{Lemma}[section]
\newtheorem{theorem}[lemma]{Theorem}
\newtheorem{corollary}[lemma]{Corollary}

\newtheorem{example}[lemma]{Example}

}

\theoremstyle{THrm}{

}

\theoremstyle{THhit}{

}

\makeatletter
\def\theequation{\arabic{equation}}

\begin{document}%

\setcounter{page}{1}

\LRH{H. Huynh and A. Kalkan}

\RRH{Pullback and forward attractors of contractive difference equations}






\subtitle{}

\title{Pullback and forward attractors of contractive difference equations}

\authorA{Huy Huynh and Abdullah Kalkan{\sf{*}}}

\affA{Universit\"at Klagenfurt,\\ Institut f\"ur Mathematik,\\ 9020 Klagenfurt am W\"orthersee, Austria\\
E-mail: pham.huynh@aau.at\\
E-mail: abukalkan@gmail.com\\
{\sf{*}}Corresponding author}

\begin{abstract}
The construction of attractors of a dissipative difference equation is usually based on compactness assumptions. In this paper, we replace them with contractivity assumptions under which the pullback and forward attractors are identical. As a consequence, attractors degenerate to unique bounded entire solutions. As an application, we investigate attractors of integrodifference equations which are popular models in theoretical ecology.
\end{abstract}

\KEYWORD{Pullback attractor, Forward attractor, Contractive mapping, Dissipative difference equation, Semilinear difference equation, Contractive difference equation, Integrodifference equation}


\begin{bio}
Huy Huynh received his MSc in Applied Mathematics at Université de Rennes 1, France in 2018. His research interests include partial difference equations, dynamical systems, fluid dynamics, numerical analysis and their applications in modelling and discretisation. At present, he is a PhD student in Applied Mathematics at Universität Klagenfurt, Austria.

\noindent
After studying mathematics at the Ludwig-Maximilians-University Munich, Germany, Abdullah Kalkan received his MSc in Mathematics at University of Augsburg, Germany and PhD in Applied Mathematics at Universität Klagenfurt, Austria. His research interests include dynamical systems, numerical analysis, stochastic processes and stability theory of stochastic differential equations.

\noindent
This research of the authors has been supported by the Austrian Science Fund (FWF) under Grant Number P 30874-N35.
\end{bio}

\maketitle

\section{Introduction}
Different attractor notions for nonautonomous difference equations has received a large amount of attention over recent years as they reflect the long-term behaviour of a process and consist of bounded entire solutions. Whilst the theory of attraction to autonomous systems is well established where attractors are given by invariant $\omega$-limit sets, the generalization to nonautonomous systems is not always appropriate. This leads us to considering nonautonomous sets that are not necessarily invariant. There are two ways to describe attractors, namely \textit{forward attractors} and \textit{pullback attractors} (\cite{Huynh:20, Kloeden:2000, Kloeden:16}). These two types are independent concepts, whereby in the autonomous case they are equivalent (\cite{Kloeden:11}).

In order to construct an attractor, one is in need of dissipativity and compactness properties. In general, nevertheless, the latter may not be easy to verify and thus can be replaced by alternative assumptions such as contractivity. Indeed, an example of a contractive difference equation which is not compact will be shown in Subsection \mbox{\ref{excontractive}}. For this reason, our goal here is to investigate the existence and structure of pullback and forward attractors under the assumption of contraction. Furthermore, based on abstract fixed-point theorems, error estimates are provided to numerically approximate attractors. In particular, we discuss two classes of dissipative difference equations and show that forward and pullback attractors degenerate to unique bounded entire solutions: first, semilinear difference equations in Banach spaces and then contractive difference equations in general complete metric spaces.

Our abstract theory can be applied to integrodifference equations in Banach spaces of continuous functions over a compact domain. The right-hand side of such equations consists of nonlinear integral operators of Hammerstein type, which in theoretical ecology usually describe the spatial dispersal and evolution of species with nonoverlapping generations (\cite{Jacobsen:15, Kot:1986, Lutscher:19}). On the other hand, as integrodifference equations are infinite-dimensional dynamical systems in the sense that one cannot implement the solutions, we are interested in their spatial discretisation to simulate the dynamical behaviours obtained in finite-dimensional state spaces. Numerical simulations require discretisation. Thus, appropriate discretisation to replace our integral equations by some corresponding methods are given by common techniques in numerical analysis, e.g., Nystr\"om method (\mbox{\cite{Atkinson:09}}).

As an example, we consider the following realistic problem in which a species lives in a habitat and an integrodifference equation is used to model it. For simplicity, the habitat here can be considered as the compact interval $\left[-\tfrac{L}{2},\tfrac{L}{2}\right]$ in one dimension and the initial condition can be chosen as any function in $C\left[-\tfrac{L}{2},\tfrac{L}{2}\right]$. Now we iterate our equation and consider the growth and distribution of the species in each generation. Due to sufficient or insufficient natural conditions, the growth or decay of the total population can be observed in different seasons. Furthermore, ecologists take certain measures to protect and promote their growth and distribution in order to prevent the species from dying out. Therefore, we use and compare various supportive measures for the total population, i.e., inhomogeneities in our equation in the last subsection, in order to determine the best possible influence on the population.

The contents of this paper are as follows: In Section 2, we first establish the necessary terminology and provide the results for the unique existence and construction of attractors of two classes of dissipative difference equations: semilinear difference equations in Banach spaces and contractive difference equations in complete metric spaces. In addition, results for the latter are based on a version of contraction principle for composite mappings. Followed by that, Section 3 addresses the related notions for integrodifference equations of Hammerstein type and presents some applications of our abstract theory.

\paragraph{Notation} Let $\Z$ be the set of all integers and $\N:=\{1,2,...\}$ the positive integers. On a metric space $(U,d)$, $I_U$ is the identity map, $B_r(x)$ and $\bar{B}_r(x)$ the open and closed balls respectively, with center $x\in U$ and radius $r>0$. We write $\dist(x,A):=\inf_{a\in A}d(x,a)$ for the distance of $x$ from a set $A\subseteq U$ and $B_r(A):=\{x\in X:\,\dist(x,A)<r\}$ for its $r$-neighborhood. The Hausdorff semidistance of bounded and closed sets $A,B \subseteq U$ is then defined as
\begin{align*}
\dist(A,B):=\sup_{a\in A}\inf_{b\in B}d(a,b). 
\end{align*}
A subset $\A \subseteq\Z\times U$ with $t$-fibers defined by $\A(t):=\{u\in U:\,(t,u)\in \A\}$ is called a nonautonomous set.

\section{Nonautonomous difference equations} \label{sec:DDE}
Let $(U,d)$ be a complete metric space. The paper deals with nonautonomous difference equations of the form
\begin{align} \label{deq}
u_{t+1}=H_t(u_t) \tag{$\Delta$}
\end{align}
with right-hand side $H_t:U\to U$. For an \textit{initial time} $\tau \in \Z$, a ~\textit{forward solution} to \eqref{deq} is a sequence $(u_t)_{\tau\leq t}$ in $U$ satisfying
\begin{align} \label{solid}
u_{t+1}\equiv H_t(u_t)
\end{align}
for all $\tau\leq t\in\Z$ while an \textit{entire solution} $(u_t)_{t\in\Z}$ satisfies \eqref{solid} on $\Z$. We define the \textit{general solution} $\varphi:\{(t,\tau,u)\in\Z^2\times U:\,\tau\leq t\}\to U$ to \eqref{deq}  by
\begin{align*}
\varphi(t,\tau,u)
:=
\begin{cases}
u,&\tau=t,\\
H_{t-1} \circ \ldots \circ H_{\tau}(u),&\tau<t.
\end{cases}
\end{align*}
By the definition of $\varphi$, the \textit{process property}
\begin{align} \label{pp}
\varphi(t,s,\varphi(s,\tau,u))=\varphi(t,\tau,u)
\end{align}
holds for all $\tau \leq s\leq t \in \Z$ and $u\in U$.
A solution $u^\ast=(u^\ast_t)_{t\in \Z}$ to \eqref{deq} is called \textit{globally attractive}, if 
\begin{align*} 
\lim_{t\rightarrow\infty} d(u^\ast_t,\varphi(t,\tau,u_\tau))=0
\end{align*}
holds for all $\tau \in \Z$ and $u_\tau \in U$.

A nonautonomous set $\A$ is called
\begin{itemize}
\item \textit{positively invariant} or \textit{invariant}, if 
\begin{align*}
H_t(\A(t))\subseteq\A(t+1)\textrm{ or } H_t(\A(t))=\A(t+1)
\end{align*}
holds for all $t\in\Z$, respectively,
\item \textit{compact}, if every $t$-fiber $\A(t)$ is compact.
\end{itemize}
Moreover, a nonautonomous set $\A$ with the two properties of invariance and compactness
%
is called a \textit{forward attractor} of \eqref{deq}, if 
\begin{align*}
\lim_{t\to\infty}\dist(\varphi(t,\tau,B),\A(t))=0
\quad\text{for all bounded }B\subseteq U
\end{align*}
and a \textit{pullback attractor} of \eqref{deq}, if 
\begin{align*}
\lim_{\tau\to-\infty}\dist(\varphi(t,\tau,B),\A(t))=0
\quad\text{for all bounded }B\subseteq U. 
\end{align*}

We denote equations \eqref{deq} as $\theta$-periodic for some $\theta\in\N$, if
\begin{align} \label{pdeq}
H_{t+\theta}=H_t
\end{align}
holds for all $t\in\Z$ and $\theta$ is consequently denoted the period. In case of $\theta=1$, i.e., $H_{t+1}=H_t=:H$ for each $t\in \Z$, we say \eqref{deq} is autonomous.

\subsection{Semilinear difference equations}
We consider nonautonomous difference equations \eqref{deq} in Banach spaces $(\tilde{U},\|\cdot\|)$. Let $L_t\in L(\tilde{U})$, $t\in\Z$, be a sequence of bounded linear operators and $K_t:\tilde{U}\to\tilde{U}$, $t\in\Z$, be mappings. We consider \eqref{deq} of the semilinear form
\begin{align} \label{semilineardeq}
H_t(u):=L_t u + K_t(u).
\end{align}
The general solution to \eqref{deq} of the form \eqref{semilineardeq} by the variation of constants formula \cite[Theorem 3.1.16, p.~100]{Poetzsche:10} is of the form

\begin{align} \label{voc}
\varphi(t,\tau,u_\tau)=\Phi(t,\tau)u_\tau+\sum_{s=\tau}^{t-1}\Phi(t,s+1)K_s\left(\varphi(s,\tau,u_\tau)\right) \fall \tau\leq t,
\end{align}
where the transition operator $\Phi:\{(t,\tau)\in\Z^2:\tau\leq t\}\to L(\tilde{U})$ is defined by
\begin{align*}
\Phi(t,\tau):=\begin{cases}
L_{t-1}\ldots L_\tau, &\tau<t,\\
I_{\tilde{U}}, &\tau=t.
\end{cases}
\end{align*}

The next lemma is helpful for proving the main result in this subsection:
\begin{lemma} \label{semilinearlem}
Let $H_t: \tilde{U}\to \tilde{U}$ be of the semilinear form \eqref{semilineardeq}. Suppose there exist reals $\kappa_t\geq 0, \alpha_t> 0$ ~for all $t\in\Z$ and $\gamma \geq 1$ such that
\begin{gather} \label{semilinearass1}
\|K_t(u)-K_t(\bar{u})\|\leq \kappa_t\|u-\bar{u}\|
\end{gather}
and
\begin{gather} \label{semilinearass2}
\|\Phi(t,\tau)\|\leq \gamma \prod_{r=\tau}^{t-1} \alpha_r 
\end{gather}
hold for all $\tau\leq t\in\Z$ and $u,\bar{u}\in \tilde{U}$, then the general solution of \eqref{semilineardeq} satisfies the estimate
\begin{align} \label{semilinearest}
\|\varphi(t,\tau,u)-\varphi(t,\tau,\bar{u})\|\leq \gamma\|u-\bar{u}\| \prod_{r=\tau}^{t-1} (\alpha_r+\gamma \kappa_r)
\end{align}
for all $\tau\leq t\in\Z$ and $u,\bar{u}\in \tilde{U}$.
\end{lemma}

\begin{proof} {Proof} 
Let $\tau\leq t\in\Z$ and $u,\bar{u}\in \tilde{U}$. The variation of constants formula \eqref{voc} and the assumptions \eqref{semilinearass1}-\eqref{semilinearass2} then imply
\begin{align*}
\|\varphi(t,\tau,u)-\varphi(t,\tau,\bar{u})\|&\leq \gamma\|u-\bar{u}\| \prod_{r=\tau}^{t-1} \alpha_r + \gamma\sum_{s=\tau}^{t-1}\kappa_s \|\varphi(s,\tau,u)-\varphi(s,\tau,\bar{u})\| \prod_{r=s+1}^{t-1} \alpha_r.
\end{align*}
This leads to the estimate
\begin{align*}
\|\varphi(t,\tau,u)-\varphi(t,\tau,\bar{u})\| \dfrac{1}{\prod_{r=\tau}^{t-1}\alpha_r} &\leq \gamma\|u-\bar{u}\| \\
& \qquad + \gamma\sum_{s=\tau}^{t-1}\dfrac{\kappa_s}{\alpha_s} \|\varphi(s,\tau,u)-\varphi(s,\tau,\bar{u})\| \dfrac{1}{\prod_{r=\tau}^{s-1}\alpha_r}.
\end{align*}
Then the discrete Gr\"onwall inequality from \cite[Proposition A.2.1(a), p.~348]{Poetzsche:10} yields the claimed estimate \eqref{semilinearest}.
\end{proof}

In order to construct pullback and forward attractors of the equation \eqref{deq} in the semilinear form \eqref{semilineardeq}, we assume \eqref{deq} is pullback absorbing, i.e., there is a nonautonomous set $\B \subseteq \Z\times \tilde U $ satisfying 
\begin{itemize}
\item there exists a real $\rho>0$ such that $\B(t) \subseteq \bar B_\rho(0)$ for all $t\in\Z$,
\item for all bounded nonautonomous subsets $C\subseteq\Z\times\tilde U$, there exists a real $S\in\N$ such that 
\begin{align*}
\varphi(t,t-s,C(t-s))\subseteq \B(t) \fall s\geq S,
\end{align*}
\item $\B$ is positively invariant. 
\end{itemize}

\begin{theorem} \label{semilinearattractor}
Let \eqref{deq} of the form \eqref{semilineardeq} satisfy the assumptions of Lemma \ref{semilinearlem}. If \eqref{deq} has a pullback absorbing set $\B$ and 
\begin{align}\label{limtau}
\lim_{\tau \to -\infty} \prod_{r=\tau}^{t-1} (\alpha_r+\gamma \kappa_r)=0
\end{align}
for a fixed $t\in\Z$, then \eqref{deq} has a unique bounded entire solution $u^\ast=(u^\ast_t)_{t\in\Z}$ and a unique pullback attractor $\A^\ast:=\{(t, u^\ast_t)\in\Z\times \tilde{U}: t\in\Z\}$. If, in addition,
\begin{align}\label{limt}
\lim_{t \to \infty} \prod_{r=\tau}^{t-1} (\alpha_r+\gamma \kappa_r)=0
\end{align}
holds for a fixed $\tau\in\Z$, then $\A^\ast$ is also the forward attractor of \eqref{deq}.
\end{theorem}

The following proof is in fact the discrete time version of the proof of \cite[Theorem 5.4, p.~35]{Kloeden:2020}.

\begin{proof} {Proof} 
Let $\B$ be a pullback absorbing set of \eqref{deq}, $\tau\in\Z$ and $t_n\leq \tau$ a monotone decreasing sequence with $t_n\to-\infty$ as $n\to\infty$ and $b_n\in \B(t_n)$ for all $n\in\N$.
\begin{itemize}
\item[(I)] Define the sequence 
\begin{align*} 
\phi_n:=\varphi(\tau,t_n,b_n)
\end{align*} 
for all $n\in\N$. It is easy to see that $\phi_n\in \B(\tau)$ due to the positive invariance of $\B$. Moreover, $(\phi_n)_{n\in\N}$ is Cauchy, i.e., for all $\epsilon>0$, there exists a real $N\in\N$ such that 
\begin{align*}
\|\phi_n-\phi_m \| < \epsilon \fall n,m \geq N.
\end{align*}
Indeed, w.l.o.g. for $m\geq n$, the process property \eqref{pp} yields
\begin{align*}
\phi_m =\varphi(\tau,t_m,b_m)=\varphi(\tau,t_n,\varphi(t_n,t_m,b_m)) =\varphi(\tau,t_n,\bar{x}_{n,m})
\end{align*}
where $\bar{x}_{n,m}:=\varphi(t_n,t_m,b_m)\in \B(t_n)$ due to the positive invariance of $\B$. Then by the definition of $\phi_n$, the estimate \eqref{semilinearest} and the triangle inequality, one obtains 
\begin{align*}
\|\phi_n-\phi_m\|&=\|\varphi(\tau,t_n, b_n)-\varphi(\tau,t_n,\bar{x}_{n,m})\|\\
&\leq \gamma\|b_n-\bar{x}_{n,m}\| \prod_{r=t_n}^{\tau-1} (\alpha_r+\gamma \kappa_r)\\
&\leq 2 \gamma \rho \prod_{r=t_n}^{\tau-1} (\alpha_r+\gamma \kappa_r).
\end{align*}
Combining this with the assumption \eqref{limtau} yields $\|\phi_n-\phi_m \| < \epsilon$ for some sufficiently large $n$. Therefore, $(\phi_n)_{n\in\N}$ is Cauchy and consequently has a unique limit $u^\ast_\tau$ due to the completeness of $\tilde U$.

\item[(II)] The sequence $(u^\ast_\tau)_{\tau\in\Z}$ constructed in step (I) satisfies $u^\ast_\tau\in \B(\tau)$ and $u^\ast_{\tau+1}=H_\tau(u^\ast_\tau)$. Hence, $(u^\ast_t)_{t\in\Z}$ is an entire solution to \eqref{deq} and bounded due to $u^\ast_t\in \B(t)\subseteq \bar B_\rho(0)$ for all $t\in\Z$.

\item[(III)] The aim of this step is to show the uniqueness of a bounded entire solution $(u^\ast_t)_{t\in\Z}$ to \eqref{deq} in $\B$.

Let $(v^\ast_t)_{t\in\Z}$ be another bounded entire solution to \eqref{deq} in $\B$. By the definition of an entire solution to \eqref{deq}, the estimate \eqref{semilinearest}, the triangle inequality and the assumption \eqref{limtau}, we obtain
\begin{align*}
\|u^\ast_t-v^\ast_t\|&= \|\varphi(t,\tau,u^\ast_\tau)-\varphi(t,\tau,v^\ast_\tau)\|\\
&\leq \gamma\|u^\ast_\tau-v^\ast_\tau\| \prod_{r=\tau}^{t-1} (\alpha_r+\gamma \kappa_r)\\
&\leq 2\gamma \rho \prod_{r=\tau}^{t-1} (\alpha_r+\gamma \kappa_r)\xrightarrow[\tau\to-\infty]{}0.
\end{align*}
Hence, $u_t^\ast=v_t^\ast$ for all $t\in\Z$, implying the uniqueness of the bounded entire solution $(u^\ast_t)_{t\in\Z}$ to \eqref{deq} in $\B$. Moreover, since a pullback attractor consists of all bounded entire solutions (cf. \cite[Corollary 1.3.4, p.~17]{Poetzsche:10}), the nonautonomous set $\A^\ast=\{(t, u^\ast_t): t\in\Z\}$ is the pullback attractor.

\item[(IV)] In this step, additionally suppose the assumption \eqref{limt} holds for one $\tau\in\Z$ and consequently for all $\tau\in\Z$. We want to prove that $\mathcal A^\ast$ is also the forward attractor to \eqref{semilineardeq}. Indeed, for all $u_\tau\in\tilde U$, by the definition of $u^\ast_t$, the estimate \eqref{semilinearest}, the triangle inequality and the assumption \eqref{limt}, one obtains
\begin{align*}
\|\varphi(t,\tau,u_\tau)-u^\ast_t\|&= \|\varphi(t,\tau,u_\tau)-\varphi(t,\tau,u^\ast_\tau)\|\\
&\leq \gamma\|u_\tau-u^\ast_\tau\| \prod_{r=\tau}^{t-1} (\alpha_r+\gamma \kappa_r)\\
&\leq 2\gamma\rho \prod_{r=\tau}^{t-1} (\alpha_r+\gamma \kappa_r) \xrightarrow[t\to\infty]{}0.
\end{align*}
Thus, $\A^\ast$ is also the forward attractor.

\end{itemize}
The proof is now set and done.
\end{proof}

Next we regard a special case of the previous results where \mbox{\eqref{deq}} of the form \mbox{\eqref{semilineardeq}} is $\theta$-periodic for some $\theta\in\N$, i.e.,
\begin{align} \label{psemilinearass}
L_{t+\theta}=L_t \textrm{ and } K_{t+\theta}=K_t
\end{align}
hold for all $t\in\Z$, as follows.
\begin{corollary} \label{corpsemilinear}
Let \eqref{deq} of the form \eqref{semilineardeq} satisfy the assumptions of Lemma \ref{semilinearlem}. If there exists a $\theta\in\N$ such that the assumptions \eqref{psemilinearass} hold for all $t\in\Z$ and
\begin{align*} 
\prod_{r=0}^{\theta-1} (\alpha_r+\gamma\kappa_r)<1,
\end{align*}
then the unique bounded entire solution $u^\ast$ to \eqref{deq} is $\theta$-periodic and the nonautonomous set $\A^\ast=\{(t, u^\ast_t)\in\Z\times \tilde{U}: t\in\Z\}$ is the pullback and forward attractor of \eqref{deq}.
\end{corollary}

\begin{proof} {Proof} 
The periodicity of $L_t$ and $K_t$ in the assumptions \eqref{psemilinearass} extends to the right-hand side $H_t$ of \eqref{deq} as well as to the constants $\kappa_t$ and $\alpha_t$ in the assumptions \eqref{semilinearass1} and \eqref{semilinearass2}, respectively. This satisfies the assumptions \eqref{limtau} and \eqref{limt} for a fixed $t\in\Z$ and a fixed $\tau\in\Z$, respectively. Hence, Theorem \ref{semilinearattractor} implies \eqref{deq} has a unique bounded entire solution $u^\ast$ and the pullback and forward attractor $\A^\ast$. Moreover, thanks to {\cite[Proposition 1.4.5, p.~22]{Poetzsche:10}},
$\{u^\ast_{t+\theta}\}=\A^\ast(t+\theta)=\A^\ast(t)=\{u_t^\ast\}$ for all $t\in\Z$. This proves the periodicity of $u^\ast$. 
\end{proof}

\subsection{Contractive difference equations}
In this subsection, we return to nonautonomous difference equations \eqref{deq} in general complete metric spaces $(U,d)$ whose right-hand side $H_t$ satisfies the following standing assumptions:
\begin{itemize}
\item [(i)] there exists a real $\lambda_t\geq 0$ such that
\begin{align*} 
d(H_t(u),H_t(\bar{u}))\leq\lambda_td(u,\bar{u}) \fall t\in\Z, u,\bar{u}\in U,
\end{align*}

\item [(ii)] there exists a sequence $(\tilde{u}_t)_{t\in\Z}$ such that $\left(H_t(\tilde{u}_t)\right)_{t\in\Z}$ is bounded for all $t\in\Z$, i.e., there exist a real $R>0$ and a point $\hat{u}\in U$ such that 
\begin{align*} 
H_t(\tilde{u}_t) \subseteq B_R(\hat{u}) \fall t\in\Z.
\end{align*}
\end{itemize}

The following lemma enables us to prove our results in this subsection. It is a variant of the contraction mapping principle {\cite[Theorem 17.1(a), p.~187]{Deimling:1985}} where the iterates $F^T$ of a mapping $F$ are defined recursively via $F^0:=I_X$ and $F^T:=F\circ F^{T-1}$ for all $T\in\Z, T>0$:
\begin{lemma} \label{contractionlem}
Let $(X,d_\infty)$ be a complete metric space and $F:X\to X$. If there exist a $T\in\N$ and a real $\ell\in[0,1)$ such that 
\begin{align*}
d_\infty(F^T(x_1),F^T(x_2))\leq\ell d_\infty(x_1,x_2)
\quad\fall x_1,x_2\in X,
\end{align*}
then $F$ possesses a unique fixed point $x^\ast\in X$. Moreover, the error estimate
\begin{align} \label{errestlem}
d_\infty(x^\ast,F^{tT}(x))\leq\dfrac{\ell^t}{1-\ell}d_\infty(x,F^T(x))
\end{align}
is satisfied for all $t\in\N$ and $x\in X$.
\end{lemma}

The Lipschitz condition for the composite mappings and the error estimate in Lemma \ref{contractionlem} applied to \eqref{deq} can be stated as follows.
\begin{theorem} \label{thmsolution}
Let the assumptions $(i)$-$(ii)$ hold. If there exists a $T\in\N$ such that $\sup_{s\in\Z}d(u,\varphi(s,s-T,u))<\infty$ for all $u\in U$ and
\begin{align} \label{ellass}
\ell:=\sup_{\tau\in\Z}\prod_{r=\tau}^{\tau+T-1} \lambda_r <1,
\end{align}
then the following results hold:
\begin{itemize}
\item [(a)] \eqref{deq} possesses a unique bounded entire solution $u^\ast=(u_t^\ast)_{t\in\Z}$. Moreover,  $u^\ast$ is globally attractive and the error estimate 
\begin{align} \label{errest0}
\sup_{s\in\Z}d(u_s^\ast,\varphi(s, s-tT,u)) \leq \dfrac{\ell^t}{1-\ell} \sup_{s\in\Z} d(u,\varphi(s, s-T, u))
\end{align}
is satisfied for all $t\in\N$ and $u\in U$,

\item [(b)]  the nonautonomous set $\A^\ast:=\{(t, u^\ast_t)\in\Z\times U: t\in\Z\}$ is the pullback and forward attractor of \eqref{deq}.
\end{itemize}
\end{theorem}

\begin{proof} {Proof} \

\begin{itemize}
\item [(a)] In order to apply Lemma \ref{contractionlem}, let us consider the space $X:=\ell^\infty(U)$ of all bounded sequences $(u_t)_{t\in\Z}$ in $U$ equipped with the metric 
\begin{align*}
d_\infty(u,v):=\sup_{t\in\Z}d(u_t,v_t)
\end{align*}
and define the mapping 
\begin{align*}
F(u)_t:=H_{t-1}(u_{t-1})
\end{align*}
for all $t\in\Z$.  Obviously, $\ell^\infty(U)$ is a complete metric space and a sequence $u^\ast=(u_t^\ast)_{t\in\Z}$ is a bounded entire solution to \eqref{deq} if and only if $u^\ast=F(u^\ast)$.

We first show that $F(u)$ is bounded for each $u\in\ell^\infty(U)$. By the assumptions (i)-(ii) and the triangle inequality, one obtains
\begin{align*}
d(H_t(u_t),\hat{u})
\leq
d(H_t(u_t),H_t(\tilde{u}_t))+d(H_t(\tilde{u}_t),\hat{u})
\leq \sup_{t\in\Z} \lambda_t d(u,\tilde{u}) +R
\end{align*}
and passing to the supremum over $t$ yields $d_\infty(F(u),\hat{u}) \leq \sup_{t\in\Z} \lambda_t d(u,\tilde{u}) +R$.
Thus, $F$ maps the bounded sequences into the bounded sequences, i.e.,\ $F:X\to X$ is well-defined.

Next, we show by induction
\begin{align} \label{FT}
F^T(u)_{\tau+T}= \varphi(\tau+T,\tau,u_{\tau}).
\end{align}
Indeed, the initial case for $T=1$ holds by the definition. We then assume the induction hypothesis \eqref{FT} holds for a particular $T\in\N.$ By the induction hypothesis, one obtains
\begin{align*}
F^{T+1}(u)_{\tau+T+1}&=[F\circ F^T(u)]_{\tau+T+1}=F(F^T(u))_{\tau+T+1}\\
&=H_{\tau+T}(F^T(u)_{\tau+T})=\varphi(\tau+T+1,\tau,u_\tau),	
\end{align*}
which implies the induction step for $T+1 \in \N$. Hence \eqref{FT} holds for all $T \in \N$.

Finally, we compute the Lipschitz constant of $F^T$ for some $T\in\N$
\begin{align*}
d(F^T(u)_{\tau +T},F^T(\bar{u})_{\tau +T})&=d(\varphi(\tau+T,\tau,u),\varphi(\tau+T,\tau,\bar{u})) \\
&\leq \lambda_{\tau+T-1} d(\varphi(\tau+T-1,\tau,u),\varphi(\tau+T-1,\tau,\bar{u})) \\
&\leq \lambda_{\tau+T-1} \ldots \lambda_{\tau} d(u,\bar{u})
=\left(\prod_{r=\tau}^{\tau+T-1} \lambda_r\right) d(u,\bar{u})\\
&\leq \left(\sup_{\tau\in\Z} \prod_{r=\tau}^{\tau+T-1} \lambda_r\right) d_\infty(u,\bar{u})
= \ell d_\infty(u,\bar{u})
\end{align*}
for all $u,\bar{u}\in \ell^\infty(U)$ and passing to the supremum over all $\tau\in\Z$ yields
\begin{align*}
d_\infty(F^T(u)_{\tau +T},F^T(\bar{u})_{\tau +T}) \leq \ell d_\infty(u,\bar{u}).
\end{align*}
Applying Lemma \ref{contractionlem} with $\ell<1$ then derives that $F$ has a unique fixed point $u^\ast$, which in turn is a unique bounded entire solution of \eqref{deq}. 

Concerning the error estimate \eqref{errest0}, take a point $\bar{u}\in U$ and consider the constant sequence $u:=(\bar{u})_{t\in\Z}\in\ell^\infty(U)$. By the definition of $d_\infty$ and the error estimate \eqref{errestlem}, one obtains
\begin{align*}
d(u_s^\ast,\varphi(s,s-tT,u))
&= d(u_s^\ast,F^{tT}(u)_s)
\leq d_\infty(u^\ast,F^{tT}(u))\\
&\leq \dfrac{\ell^t}{1-\ell}d_\infty(u,F^T(u))
= \dfrac{\ell^t}{1-\ell} \sup_{s\in\Z}d(u,\varphi(s, s-T, u))
\end{align*}
for all $s \in\Z$ and passing to the supremum over all $s \in\Z$ proves \eqref{errest0}. Furthermore, it is not difficult to see that $u^\ast$ is also globally attractive.

\item [(b)] First, we verify the following properties:
\begin{itemize}
\item [$\bullet$] $\A^\ast$ is invariant: $\A^\ast(t+1)=\{H_t(u^\ast_t)\}=H_t\left(\{u^\ast_t\}\right)=H_t\left(\A^\ast(t)\right)$ for all $t\in \Z$,

\item [$\bullet$] $\A^\ast$ is compact: $\A^\ast(t)=\{u^\ast_t\}$ is a singleton and thus compact.
\end{itemize} 

Now, let $B\subseteq U$ be bounded and $(\tau,u_\tau) \in \Z \times B $, i.e., there exist a real $R>0$ and a point $u\in U$ so that $u_\tau, u^\ast_\tau \in  B_R(u)$. The definition of $\A^\ast$, the triangle inequality and the assumptions (i)-(ii) derive
\begin{align*}
\inf_{a \in \A^\ast(t)} d(\varphi(t,\tau,u_\tau),a)&= d(\varphi(t,\tau,u_\tau),u_t^\ast)= d(\varphi(t,\tau,u_\tau),\varphi(t,\tau,u_\tau^\ast)) \\
&\leq \left(\prod_{r=\tau}^{t-1}\lambda_s\right)d(u_\tau,u_\tau^\ast) \leq 2R\prod_{r=\tau}^{t-1}\lambda_s
\end{align*}
for all $u_\tau \in B$ and passing to the supremum over all $u_\tau \in B$ yields
\begin{align*} 
\dist(\varphi(t,\tau,B),\A^\ast(t))\leq 2R\prod_{r=\tau}^{t-1}\lambda_s.
\end{align*}
Hence, combining this with the limit relations $ \lim_{t \rightarrow \infty}\prod_{r=\tau}^{t-1}\lambda_s=\lim_{\tau \rightarrow -\infty}\prod_{r=\tau}^{t-1}\lambda_s=0$ implies $\A^\ast$ is the pullback and forward attractor of \eqref{deq}.
\end{itemize}

This completes the proof of the theorem.
\end{proof}

We next formulate the result of periodic attractors as follows.
\begin{corollary} \label{patt}
If there exists a $\theta\in\N$ such that \eqref{pdeq} holds for all $t\in\Z$ and
\begin{align} \label{pprod}
\prod_{r=0}^{\theta-1}\lambda_r<1,
\end{align}
then the unique bounded entire solution $u^\ast$ to \eqref{deq} is $\theta$-periodic and the nonautonomous set $\A^\ast=\{(t, u^\ast_t)\in\Z\times U: t\in\Z\}$ is the pullback and forward attractor of \eqref{deq}.
\end{corollary}
\begin{proof} {Proof} 
The assumption \eqref{pdeq} yields both \eqref{deq} and $\lambda_t$ are also $\theta$-periodic, implying the assumption \eqref{ellass} is satisfied for $T:=\theta$. The rest of the proof is verbatim to that in the proof of Corollary \ref{corpsemilinear} where Theorem \ref{thmsolution} is used instead.
\end{proof}

\subsection{Examples of contractive non-compact difference equations} \label{excontractive}
In this subsection, we provide an example where a difference equation is contractive but not compact in order to explain the reason behind our alternative assumption. In Banach spaces $C\left[-\tfrac{L}{2},\tfrac{L}{2}\right]$ with
\begin{align} \label{norm}
\|v\|&:=\sup_{|x|\leq\tfrac{L}{2}}|v(x)| \fall v\in C\left[-\tfrac{L}{2},\tfrac{L}{2}\right],
\end{align}
we consider nonautonomous difference equations \eqref{deq}  with right-hand side given by
\begin{align*} 
H_t(u)(x):=\dfrac{b_t(x) u(x)}{1+|u(x)|} 
\end{align*}
where $b_t:\left[-\tfrac{L}{2},\tfrac{L}{2}\right]\to[0,\infty)$ is continuous and $h_t\in C\left[-\tfrac{L}{2},\tfrac{L}{2}\right]$. It is obvious that $H_t$ is well-defined in $C\left[-\tfrac{L}{2},\tfrac{L}{2}\right]$, i.e., $H_t:C\left[-\tfrac{L}{2},\tfrac{L}{2}\right]\to C\left[-\tfrac{L}{2},\tfrac{L}{2}\right]$. However, $H_t$ may not be completely continuous. Indeed, if $H_t$ is completely continuous, i.e., $H_t$ maps any bounded set of $\left[-\tfrac{L}{2},\tfrac{L}{2}\right]$ into a relatively compact subset of $\left[-\tfrac{L}{2},\tfrac{L}{2}\right]$, then $\overline{H_t(B_1(0))}$ is compact. This implies that $\overline{H_t(B_1(0))}$ contains a closed ball $\overline{B}_r(0)$ with some $r>0$. Therefore, by a consequence of Riesz's lemma \cite[Theorem 4, p.~3]{Diestel:1984}, the dimension of the state space $\dim C\left[-\tfrac{L}{2},\tfrac{L}{2}\right]$ is finite, which does not hold. Hence, the operator $H_t$ is not compact.

On the other hand, it can be easily seen that $H_t$ is contractive, i.e., $H_t$ satisfies Assumption (i) in Subsection 2.2. where $d(u,\bar{u}):=\|u-\bar{u}\|$ for all $u,\bar{u}\in C\left[-\tfrac{L}{2},\tfrac{L}{2}\right]$ and $\lambda_t:=\sup_{|x|\leq \tfrac{L}{2}}b_t(x)$ for all $t\in\Z$. Indeed, letting $u,\bar{u}\in C\left[-\tfrac{L}{2},\tfrac{L}{2}\right]$, one obtains
\begin{align*}
|H_t(u)(x)-H_t(\bar{u})(x)|&=b_t(x)\left|\dfrac{u}{1+|u(x)|}-\dfrac{\bar{u}}{1+|\bar{u}(x)|}\right|\\
&\leq b_t(x) |u(x)-\bar{u}(x)|\leq\lambda_t\|u-\bar{u}\|.
\end{align*}
and passing to the supremum over all $x\in\left[-\tfrac{L}{2},\tfrac{L}{2}\right]$ then yields the contractivity of $H_t$.

In conclusion, one can construct attractors of nonautonomous difference equations under the assumption of contractivity instead of compactness.

\section{Applications} \label{sec:App}
Now, we apply the theory from the previous section to integrodifference equations (shortly IDEs) as applications. To be more precise, we study a class of IDEs involving Hammerstein integral operators satisfying a global Lipschitz condition and being defined on the space of continuous functions over a compact domain.

\subsection{Integrodifference equations}
In this part, we deal with scalar IDEs of Hammerstein type which are equations $\eqref{deq}$ with right-hand side given by
\begin{align} \label{hdef}
H_t(u):=\int_{-\tfrac{L}{2}}^{\tfrac{L}{2}}k_t(\cdot,y)g_t(y,u(y)) \d y+h_t
\end{align}
under the following standing assumptions:
\begin{itemize}
\item the kernel $k_t:\left[-\tfrac{L}{2},\tfrac{L}{2}\right]^2\to\R$ is continuous and satisfies
\begin{align*} 
\sup_{|x|\leq\tfrac{L}{2}}\int_{-\tfrac{L}{2}}^{\tfrac{L}{2}}|k_t(x,y)|\d y<\infty,
\end{align*}

\item the growth function $g_t:\left[-\tfrac{L}{2},\tfrac{L}{2}\right]\times\R\to\R$ is such that $g_t$ is bounded 
and satisfies a global Lipschitz condition for all $t\in\Z$, i.e.,\ there exists a continuous function $\hat{g}_t:\left[-\tfrac{L}{2},\tfrac{L}{2}\right]\to[0,\infty)$ such that for all $x\in\left[-\tfrac{L}{2},\tfrac{L}{2}\right]$ and $z,\bar{z}\in\R$, one has
\begin{align*} 
|g_t(x,z)-g_t(x,\bar{z})|\leq\hat{g}_t(x)|z-\bar{z}| \fall t\in\Z;
\end{align*}
additionally, $g_t(\cdot,z):\left[-\tfrac{L}{2},\tfrac{L}{2}\right]\to\R$ is continuous and $g_t(x,0)=0$ for all $t\in\Z$ and $x\in\left[-\tfrac{L}{2},\tfrac{L}{2}\right]$,

\item the inhomogeneity $h_t\in C\left[-\tfrac{L}{2},\tfrac{L}{2}\right]$ satisfies $\sup_{t\in\Z}\|h_t\|<\infty$ where the norm $\|\cdot\|$ is defined in \eqref{norm}.
\end{itemize}

The right-hand side \eqref{hdef} of \eqref{deq} defines a Hammerstein integral operator whose properties are summarized in the following theorem {\cite[Theorem B.5 and Corollary B.6]{Poetzsche:19}}:
\begin{theorem}\label{thmlip}
Let $t\in\Z$. The Hammerstein operator $H_t:C\left[-\tfrac{L}{2},\tfrac{L}{2}\right]\to C\left[-\tfrac{L}{2},\tfrac{L}{2}\right]$ is well-defined and satisfies the global Lipschitz condition
\begin{align} \label{LipHt}
\|H_t(u)-H_t(\bar{u})\|
\leq
\sup_{|x|\leq \tfrac{L}{2}}\int_{-\tfrac{L}{2}}^{\tfrac{L}{2}}|k_t(x,y)|\hat{g}_t(y) \d y
\|u-\bar{u}\|
\end{align}
for all $u,\bar{u}\in C\left[-\tfrac{L}{2},\tfrac{L}{2}\right]$; one obtains
\begin{align*}
\lambda_t:=\sup_{|x|\leq\tfrac{L}{2}}\int_{-\tfrac{L}{2}}^{\tfrac{L}{2}}|k_t(x,y)|\hat{g}_t(y) \d y 
\end{align*}
as the Lipschitz constant of the Hammerstein operator $H_t$. 
\end{theorem}

Notice that the Lipschitz constant $\lambda_t$ in Theorem \ref{thmlip} is the same as the constant $\lambda_t$ in Assumption (i) in Subsection 2.2.

Under the assumptions of Theorem \ref{thmlip}, we obtain from Theorem \ref{thmsolution}:
\begin{theorem} \label{thmide}
Let \eqref{deq} have the right-hand side \eqref{hdef}. If there exists a $T\in\N$ such that \eqref{ellass} holds, then the following results hold:
\begin{itemize}
\item [(a)] \eqref{deq} possesses a unique bounded entire solution $u^\ast=(u_t^\ast)_{t\in\Z}$ in $C\left[-\tfrac{L}{2},\tfrac{L}{2}\right]$. Moreover, $u^\ast$ is globally attractive and the following error estimate 
\begin{align} \label{errest}
\sup_{s\in\Z}\|u_s^\ast-\varphi(s,s-tT,u)\| \leq \dfrac{\ell^t}{1-\ell} \sup_{s\in\Z}\|u-\varphi(s,s-T,u)\|
\end{align}
is satisfied for all $t\in \N$ and $u \in C\left[-\tfrac{L}{2},\tfrac{L}{2}\right]$,

\item [(b)] the nonautonomous set $\A^\ast:=\left\{(t,u^\ast_t)\in \Z\times C\left[-\tfrac{L}{2},\tfrac{L}{2}\right]: t\in\Z\right\}$ is the pullback and forward attractor of \eqref{deq} with right-hand side \eqref{hdef}.
\end{itemize}

\end{theorem}

\begin{proof} {Proof} 
First, notice that the well-definedness and Lipschitz properties of the Hammerstein type mapping $H_t$ as in \eqref{hdef} are stated in Theorem \ref{thmlip}. Then applying Theorem \ref{thmsolution} with $U:=C\left[-\tfrac{L}{2},\tfrac{L}{2}\right]$ and the metric $d$ as the norm $\|\cdot\|$ proves both results of the theorem.
\end{proof}

Next we deal with IDEs being $\theta$-periodic for some $\theta\in\N$, i.e.,
\begin{align} \label{pide}
k_{t+\theta}=k_t, g_{t+\theta}=g_t \textrm{ and } h_{t+\theta}=h_t
\end{align}
hold for all $t\in\Z$. The periodicity of the pullback and forward attractor of a $\theta$-periodic IDE is stated as follows.
\begin{corollary} \label{pattide}
Let \eqref{deq} have the right-hand side \eqref{hdef}. If there exists a $\theta\in\N$ such that the assumptions \eqref{pprod} and \eqref{pide} hold for all $t\in\Z$, then the unique bounded entire solution $u^\ast$ to \eqref{deq} is $\theta$-periodic and the nonautonomous set $\A^\ast=\{(t, u^\ast_t)\in\Z\times C\left[-\tfrac{L}{2},\tfrac{L}{2}\right]: t\in\Z\}$ is the pullback and forward attractor of \eqref{deq}. 
\end{corollary}
\begin{proof} {Proof} 
The argument is verbatim to the one given in the proof of Corollary \ref{patt} where the assumption \eqref{pide} and Theorem \ref{thmide} are used instead.
\end{proof}

\subsection{Examples}
In this subsection, we apply our theoretical results above to IDEs. IDEs are investigated under commonly used kernels $k_t$ and growth functions $g_t$ listed in Tables~\ref{tabkernel} and \ref{tabgrowth}, respectively (cf.~\cite{Lutscher:19}), whose coefficients $a_t$ and $b_t:\left[-\tfrac{L}{2},\tfrac{L}{2}\right]\to[0,\infty)$ satisfy $a_t> 0$ and $\sup_{|x|\leq \tfrac{L}{2}}|b_t(x)|\leq \beta_t$, respectively, for all $t\in\Z$ and some real $\beta_t$ with $\sup_{t\in\Z} \beta_t<\infty$.

\begin{table}[H]
\begin{center}
\begin{tabular}{c|c|c|c}
& Laplace & Gau{\ss} & tent \\ 
\hline \rule{0pt}{20pt}
$k_t(x,y)$ & $\dfrac{a_t}{2}e^{-a_t|x-y|}$ & $\dfrac{a_t}{\sqrt{\pi}}e^{-a_t^2(x-y)^2}$ & $\max(0,a_t-a_t^2|x-y|)$\\
\hline \rule{0pt}{20pt}
$\sup_{|x|\leq\tfrac{L}{2}}\int_{-\tfrac{L}{2}}^{\tfrac{L}{2}}|k_t(x,y)|\d y$ & $1-e^{-\tfrac{a_tL}{2}}$ & $\erf\left(\dfrac{a_tL}{2}\right)$ & $a_tL-\dfrac{1}{4}a_t^2L^2$\\
\end{tabular}
\caption{Kernels $k_t$ and their upper bounds $\sup_{|x|\leq\tfrac{L}{2}}\int_{-\tfrac{L}{2}}^{\tfrac{L}{2}}|k_t(x,y)|\d y$}
\label{tabkernel}
\vspace{-20pt}
\end{center}
\end{table}

\begin{table}[H]
\begin{center}
\begin{tabular}{c|c|c|c}
& logistic & Beverton-Holt & Ricker \\ 
\hline \rule{0pt}{20pt}
$g_t(x,z)$ & $\max(0,b_t(x)z(1-z))$ & $\dfrac{b_t(x)z}{1+|z|}$ & $ze^{-b_t(x)|z|}$\\
\hline \rule{0pt}{20pt}
$\sup_{|x|\leq\tfrac{L}{2}}|g_t(x,z)|$ & $\dfrac{1}{4}\beta_t$ & $\beta_t$ & $e^{-1}\beta_t$\\
\end{tabular}
\caption{Growth functions $g_t$ and their upper bounds $\sup_{|x|\leq\tfrac{L}{2}}|g_t(x,z)|$}
\label{tabgrowth}
\vspace{-20pt}
\end{center}
\end{table}

In order to construct and numerically approximate the pullback and forward attractor $\A^\ast=\left\{(t,u^\ast_t)\in \Z\times C\left[-\tfrac{L}{2},\tfrac{L}{2}\right]: t\in\Z\right\}$ of an IDE \eqref{deq} from Theorem \ref{thmide}, the Lipschitz constant $\lambda_t$ of the Hammerstein operator $H_t$ from Theorem \ref{thmlip} and the error estimate \eqref{errest} are required. First, for the Lipschitz function $\lambda_t$, one can obtain it from Table~\ref{tabgrowth}. Indeed, notice that all the growth functions $g_t$ from Table~\ref{tabgrowth} possess the same continuous function $\hat{g}_t(x):=\beta_t$ (independent of $x$) and thus
\begin{align*}
\lambda_t=\beta_t\sup_{|x|\leq\tfrac{L}{2}}\int_{-\tfrac{L}{2}}^{\tfrac{L}{2}}|k_t(x,y)|\d y.
\end{align*}
Then, concerning the error estimate \eqref{errest}, we compute $\sup_{s\in\Z}\|u-\varphi(s,s-T,u)\|$ by
\begin{align*}
\|u-\varphi(s,s-T,u)\|
&\leq \|u\|+\|\varphi(s,s-T,u)\|\\
&=\|u\|+\|H_{s-1}\left(\varphi(s-1,s-T,u)\right)\|\\
&\leq \|u\|+\sup_{s\in\Z}\|H_{s-1}\left(\varphi(s-1,s-T,u)\right)\|
\end{align*}
for all $s\in\Z$ by the triangle inequality and the definition of $\varphi(s,s-T,u)$. Followed by that, we determine the upper bound of $\sup_{s\in\Z}\|H_{s-1}\left(\varphi(s-1,s-T,u)\right)\|$. For each function $u\in C\left[-\tfrac{L}{2},\tfrac{L}{2}\right]$, one obtains
\begin{align*}
|H_s(u)(x)|
&\leq
\int_{-\tfrac{L}{2}}^{\tfrac{L}{2}}|k_s(x,y)||g_s(y,u(y))|\d y+|h_s(x)|\\
&\leq
\sup_{|x|\leq\tfrac{L}{2}}\int_{-\tfrac{L}{2}}^{\tfrac{L}{2}}|k_s(x,y)|\d y\sup_{|x|\leq\tfrac{L}{2}}|g_s(x,u(x))|+\|h\|_\infty\\
&=:l_1(s,u)+\|h\|_\infty
\end{align*}
for all $s\in\Z$ and $x\in\left[-\tfrac{L}{2},\tfrac{L}{2}\right]$ where $\|u\|_\infty:=\sup_{t\in\Z}\|u_t\|$ for all $u\in \ell^\infty(C\left[-\tfrac{L}{2},\tfrac{L}{2}\right])$. If $\sup_{s\in\Z}l_1(s,u)<\infty$, then subsequent passing to the supremum over all $x\in\left[-\tfrac{L}{2},\tfrac{L}{2}\right]$ and $s\in\Z$ yield
\begin{align*}
\sup_{s\in\Z}\|H_s(u)\| \leq \sup_{s\in\Z}l_1(s,u) + \|h\|_\infty
\end{align*}
since $l_1(s,u)$ is independent of $x$ and $h_s$ is bounded above in $s\in\Z$. This implies the upper bounds of $\sup_{s\in\Z}\|H_{s-1}\left(\varphi(s-1,s-T,u)\right)\|$ and $\|u-\varphi(s,s-T,u)\|$ for all $s\in\Z$. Similarly, passing to the supremum over all $s\in\Z$ results in
\begin{align} \label{l2}
\sup_{s\in\Z}\|u-\varphi(s,s-T,u)\|\leq \|u\| + \sup_{s\in\Z}l_1\left(s-1,\varphi(s-1,s-T,u)\right) + \|h\|_\infty=:l_2
\end{align}
and the error estimate \eqref{errest} is obtained as a result.

Let us now be more precise by illustrating some results using concrete IDEs and their discretisations. In order to avoid the numerical integration of the integral operator and obtain a full discretisation, we work with the Nystr\"om method (cf.~\cite{Atkinson:09}) by replacing the integral with a weighted sum such that the discretisation of \eqref{deq} leads to a sequence of equations $(\Delta_n)_{n\in\N}$ of the form
\begin{align*}
u_{t+1}(x)=H^n_t(u_t)(x)=\sum_{i=0}^{n}\omega_i k_t(x,\eta_i)g_t(\eta_i,u_t(\eta_i))\d y + h_t(x) \tag{$\Delta_n$}
\end{align*}
for all $t\in\Z$ and $x\in\left[-\tfrac{L}{2},\tfrac{L}{2}\right]$, assuming given a convergent quadrature rule with nodes $\eta_i$ and weights $\omega_i$ for $i=0,1,\ldots,n$. In particular, \mbox{\egref{example}} below uses the trapezoidal quadrature rule to approximate integrals with
\begin{align*}
h:=\dfrac{L}{n}, \eta_i:=-\dfrac{L}{2}+ih \textrm{ and } \omega_i:=\begin{cases}
\dfrac{h}{2}, &i\in\{0,n\},\\
h, &\textrm{otherwise}
\end{cases}
\end{align*}
for all $i=0,1,\ldots,n$. Moreover, we can also approximate the total population $\overline{u}_t$ at time $t$ as
\begin{align*} 
\overline{u}_t:=\int_{-\tfrac{L}{2}}^{\tfrac{L}{2}} u_t(y)\d y\approx\sum_{i=0}^{n}\omega_i u_t(\eta_i).
\end{align*}
On the other hand, for a discretisation, the error estimate $\sup_{s\in\Z}\|u^\ast_s-\varphi(s,s-tT,u)\|$ in \eqref{errest} has to be less than or equal to some real tolerance $tol>0$ for some sufficiently large $t\in\Z$. Therefore, in order to meet the assumption \eqref{ellass}, one can pick $\ell \leq \dfrac{1}{2}$ and hence by \eqref{l2}
\begin{align*}
\dfrac{\ell^t}{1-\ell} \sup_{s\in\Z}\|u-\varphi(s,s-T,u)\| \leq \dfrac{\left(\tfrac{1}{2}\right)^t}{1-\tfrac{1}{2}} l_2 = 2^{1-t} l_2 < tol,
\end{align*}
resulting in $t\geq \log_2\left(\dfrac{2l_2}{tol}\right)$. Thus, we need at least 
\begin{align} \label{fewestiteration}
S:=T\log_2\left(\dfrac{2l_2}{tol}\right)
\end{align}
iterations so that the error estimate \eqref{errest} is satisfied and $\A^\ast$ is consequently the pullback and forward attractor of the IDE \eqref{deq}.

\begin{example} \label{example}
Suppose in a habitat, a species initially lives mostly on the boundary of the habitat and its population gradually decreases in the middle areas (cf. a choice of a constant initial solution $u_0$ below). Over four seasons of each year, it periodically breeds and dies as well as spreads from one place to another. Not to mention that the total population decays dramatically in fall and winter before slightly flourishing again in spring and summer every year resulting from natural conditions and the species may come to extinction after a period of time. Therefore, ecologists would like to improve the total population by supporting it in each season with various options and thus choosing the optimal one. Followed by that, a mathematical model using an IDE may help to study the system and give the optimal solution in terms of the total population of the species as a result.

For simplicity, we consider the scalar, $\theta$-periodic and inhomogeneous IDE 
\begin{align} \label{BH}
u_{t+1}(x)=H_t(u_t)(x):=\int_{-\tfrac{L}{2}}^{\tfrac{L}{2}} \dfrac{a}{2}e^{-a|x-y|}\dfrac{b_t(y)u_t(y)}{1+|u_t(y)|}\d y+h_t(x)
\end{align}
for all $(x,t)\in\left[-\tfrac{L}{2},\tfrac{L}{2}\right]\times\Z$ that fits into the realistic problem and the framework of \eqref{hdef} where
\begin{itemize}
\item the domain $\left[-\tfrac{L}{2},\tfrac{L}{2}\right]$ represents the habitat,

\item the kernel $k(x,y):=\dfrac{a}{2}e^{-a|x-y|}$ with $a>0$ represents the dispersal rate of the species from point $x$ to point $y$,

\item the periodic growth function $g_t(x,z):=\dfrac{b_t(x)z}{1+|z|}$ represents the natural growth rate of the species at point $x$ and day $t$ starting from the population $z$, while the periodic parameter $b_t(x):=\alpha_t\hat{b}(x)$ represents the natural conditions at point $x$ and day $t$ with
\begin{itemize}
\item [$\cdot$] $\alpha_t:=C\left(1+\tfrac{1}{2}\sin\left(\tfrac{2\pi t}{\theta}\right)\right)>0$ for all $t\in\Z$,
\item [$\cdot$] $\hat{b}(x)\geq 0$ for all $x\in \left[-\tfrac{L}{2},\tfrac{L}{2}\right]$,
\item [$\cdot$] $\sup_{|x|\leq \tfrac{L}{2}}|b_t(x)|\leq \beta_t:=\alpha_t \sup_{|x|\leq \tfrac{L}{2}}|\hat{b}(x)|$ for all $t\in\Z$,
\end{itemize}

\item the periodic inhomogeneity $h_t(x)$ represents the supported population at point $x$ and day $t$.
\end{itemize}

With the choice of $C$ as follows
\begin{align*}
C:=\dfrac{1}{1-e^{-\tfrac{aL}{2}}}\sqrt[\theta]{\dfrac{1}{2\prod_{r=0}^{\theta-1}\beta_r\left(1+\tfrac{1}{2}\sin\left(\tfrac{2\pi r}{\theta}\right)\right)}} >0
\end{align*}
for all $t\in\Z$, the assumptions \eqref{ellass} and \eqref{LipHt} hold where $\lambda_t:=\beta_t\left(1-e^{-aL/2}\right)$ and
\begin{align*}
\|H_t(u)-H_t(\bar{u})\|\leq \beta_t\left(1-e^{-aL/2}\right) \|u-\bar{u}\| \fall u,\bar{u}\in C(\left[-\tfrac{L}{2},\tfrac{L}{2}\right]).
\end{align*}
By Theorem \ref{thmide}, for $\ell=\prod_{r=0}^{\theta-1}\lambda_r\leq\dfrac{1}{2}$, \eqref{BH} is contractive and has the $\theta$-periodic pullback and forward attractor $\A^\ast=\{(t,u^\ast_t)\in\Z\times C(\left[-\tfrac{L}{2},\tfrac{L}{2}\right]): t\in\Z\}$ due to Corollary~\ref{pattide}. Furthermore, for the initial solution $u_0\in C(\left[-\tfrac{L}{2},\tfrac{L}{2}\right])$ representing the initial population, the error estimate \eqref{errest} then becomes
\begin{align*}
\sup_{s\in\Z}\|u_s^\ast-\varphi(s,s-t\theta,u_0)\| \leq 2^{1-t} \left(\|u_0\| + (1-e^{-aL/2})\sup_{s\in\Z}\beta_{s-1} + \|h\|_\infty\right)
\end{align*}
for all $t\in\Z$.

By the Nystr\"om method with the trapezoidal quadrature rule, the $t$-fibers $\A^\ast(t)$ are illustrated in Figure \ref{figattractor} for $t\in[0,T']$, $L:=6$, $\theta:=365$ (a year), $tol:=10^{-6}$, $a:=10$, $\hat{b}(x):=2|x|+3$ and 
\begin{align*}
u_0(x):=\begin{cases}
2x^2+0.5, &x\in[-1,1],\\
2.5, &\textrm{otherwise}.
\end{cases}
\end{align*}
For $T'$, in order to observe the periodicity of $\A^\ast$, we pick $T':=366$ (one year and one day) so that the same behaviour of the system at days 0 and 365 can be displayed. For $h_t(x)$, we consider four functions $h^i_t(x), i\in\{1,2,3,4\}$, in the following table:
\begin{table}[H]
\begin{center}
\renewcommand{\arraystretch}{2.5}
\begin{tabular}{c|c|c|c|c}
& $t\in \left(\left.0,\dfrac{\theta}{4}\right.\right]$
& $t\in \left(\left.\dfrac{\theta}{4},\dfrac{\theta}{2}\right.\right]$
& $t\in \left(\left.\dfrac{\theta}{2},\dfrac{3\theta}{4}\right.\right]$
& $t\in \left(\left.\dfrac{3\theta}{4},\theta\right.\right]$
\\ 
\hline 
$h_t^1(x)$
& $\cos\left(\dfrac{\pi x}{L}\right)$
& $\cos\left(\dfrac{\pi x}{L}\right)$
& $2\cos\left(\dfrac{\pi x}{L}\right)$
& $2\cos\left(\dfrac{\pi x}{L}\right)$
\\
\hline 
$h_t^2(x)$
& $\cos\left(\dfrac{\pi x}{L}\right)$
& $2\cos\left(\dfrac{\pi x}{L}\right)$
& $\cos\left(\dfrac{\pi x}{L}\right)$
& $2\cos\left(\dfrac{\pi x}{L}\right)$
\\
\hline 
$h_t^3(x)$
& $2\cos\left(\dfrac{\pi x}{L}\right)$
& $2\cos\left(\dfrac{\pi x}{L}\right)$
& $\cos\left(\dfrac{\pi x}{L}\right)$
& $\cos\left(\dfrac{\pi x}{L}\right)$
\\
\hline 
$h_t^4(x)$
& $2\cos\left(\dfrac{\pi x}{L}\right)$
& $\cos\left(\dfrac{\pi x}{L}\right)$
& $2\cos\left(\dfrac{\pi x}{L}\right)$
& $\cos\left(\dfrac{\pi x}{L}\right)$
\end{tabular}
\caption{Inhomogeneities $h_t^i$, $i\in\{1,2,3,4\}$}
\label{tabinhom}
\vspace{-20pt}
\end{center}
\end{table}

In Table \ref{tabinhom}, all intervals $\left(\left.0,\tfrac{\theta}{4}\right.\right]$, $\left(\left.\tfrac{\theta}{4},\tfrac{\theta}{2}\right.\right]$, $\left(\left.\tfrac{\theta}{2},\tfrac{3\theta}{4}\right.\right]$ and $\left(\left.\tfrac{3\theta}{4},\theta\right.\right]$ represent four seasons spring, summer, fall and winter, respectively, and each $h_t^i$, $i\in\{1,2,3,4\}$, represents different behaviour in each season. For instance, the populations regarding $h_t^1$ in fall and winter are twice those in spring and summer. Last, by the definitions of $u_0$, $b_t$ and $h_t$, $S=8760$ due to \eqref{fewestiteration}.

Finally, in order to meet the optimal choice for the artificial support, we compute the mean total populations $\dfrac{1}{\theta}\sum_{t=0}^{\theta-1} \bar{u}_t$ over a year for inhomogeneities $h^i_t(x), i\in\{1,2,3,4\}$, from Table \ref{tabinhom} as follows.

\begin{table}[H]
\begin{center}
\renewcommand{\arraystretch}{2}
\begin{tabular}{c|c|c|c|c}
& $h^1_t(x)$ & $h^2_t(x)$ & $h^3_t(x)$ & $h^4_t(x)$ \\
\hline
$\dfrac{1}{\theta}\sum_{t=0}^{\theta-1} \bar{u}_t$ & $7.9640$ & $5.8614$ & $8.0794$ & $10.1816$
\end{tabular}
\caption{Mean total populations $\dfrac{1}{\theta}\sum_{t=0}^{\theta-1} \bar{u}_t$ over a year for inhomogeneities $h^i_t(x), i\in\{1,2,3,4\}$}
\label{tabmeantotalpop}
\vspace{-20pt}
\end{center}
\end{table}
According to Table \ref{tabmeantotalpop}, we can conclude $h^4_t(x)$ is the best option among four. To study this behaviour, one can see that if the populations supported in spring and fall are larger than those in summer and winter to encounter natural conditions, then species can develop to the highest amount after some generations. To be more precise, in winter, the total population witnesses a significant decline and hence requires a larger support in spring. Meanwhile, despite reaching a peak in summer, the total population is supported more in fall in order to avoid a dramatic decline due to natural conditions.

\begin{figure}[H]
\begin{center}
\includegraphics[trim=125 20 155 105,clip,width=.75\textwidth]{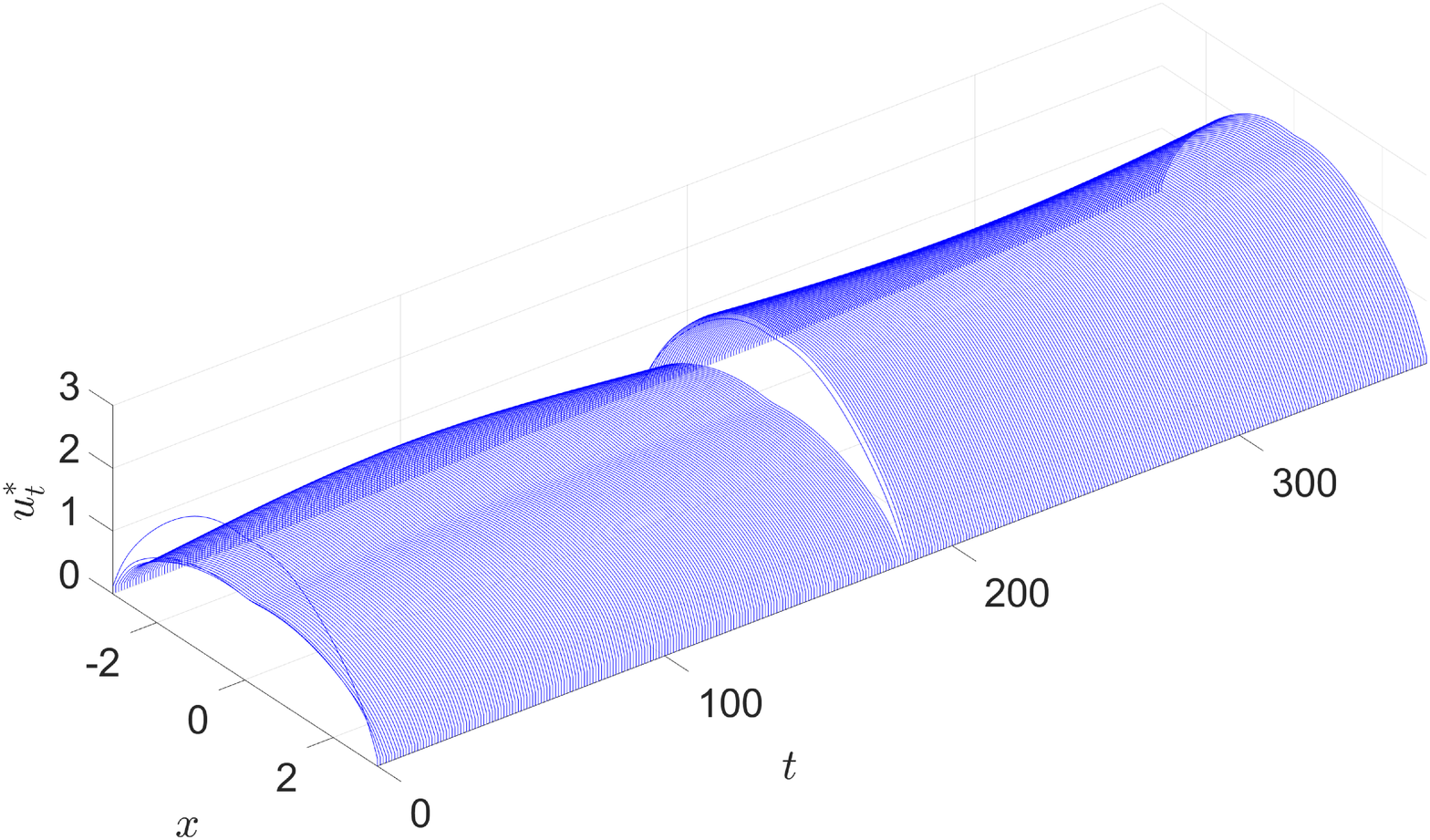}
\includegraphics[trim=125 20 155 105,clip,width=.75\textwidth]{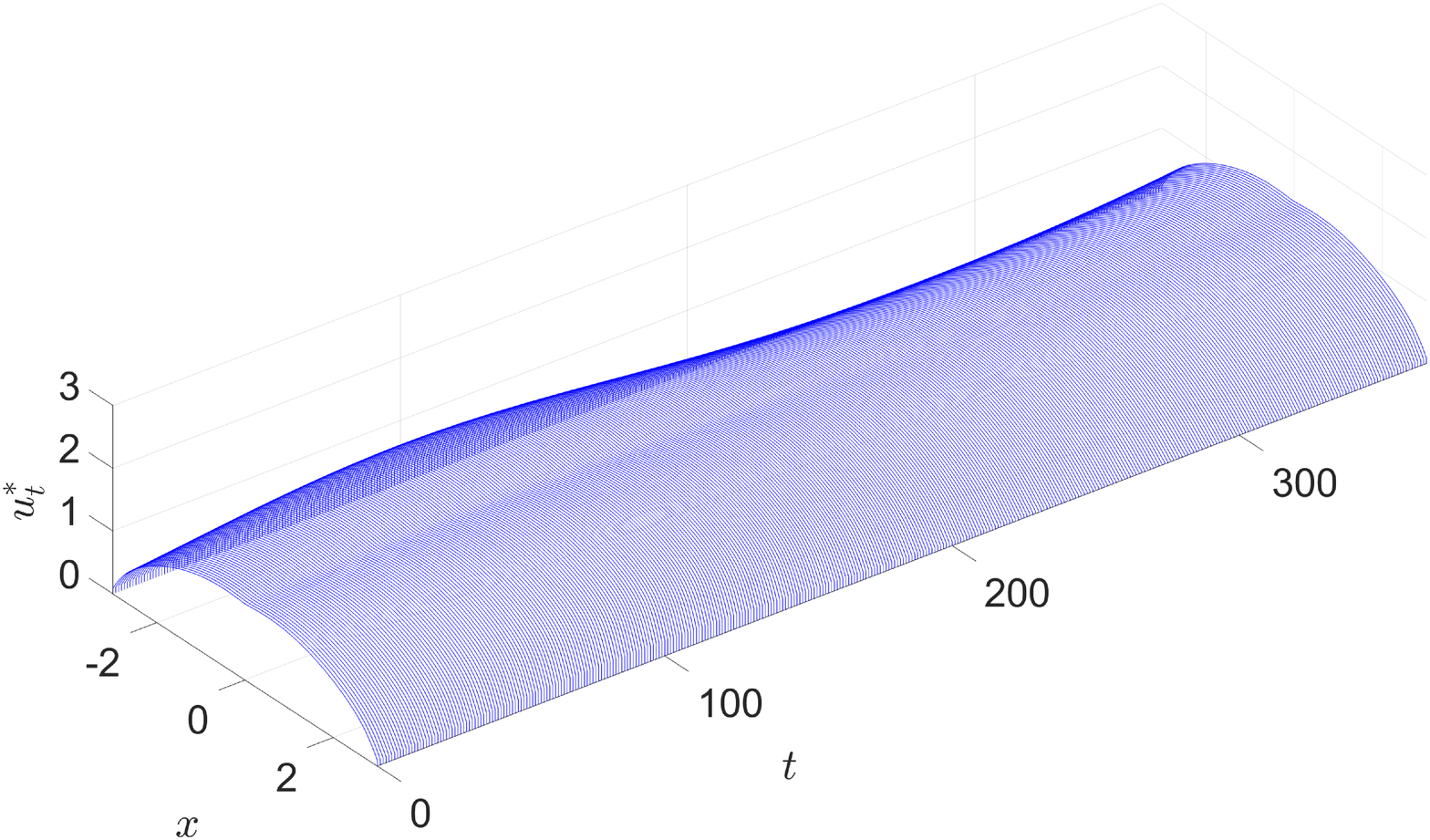}
\includegraphics[trim=125 20 155 105,clip,width=.75\textwidth]{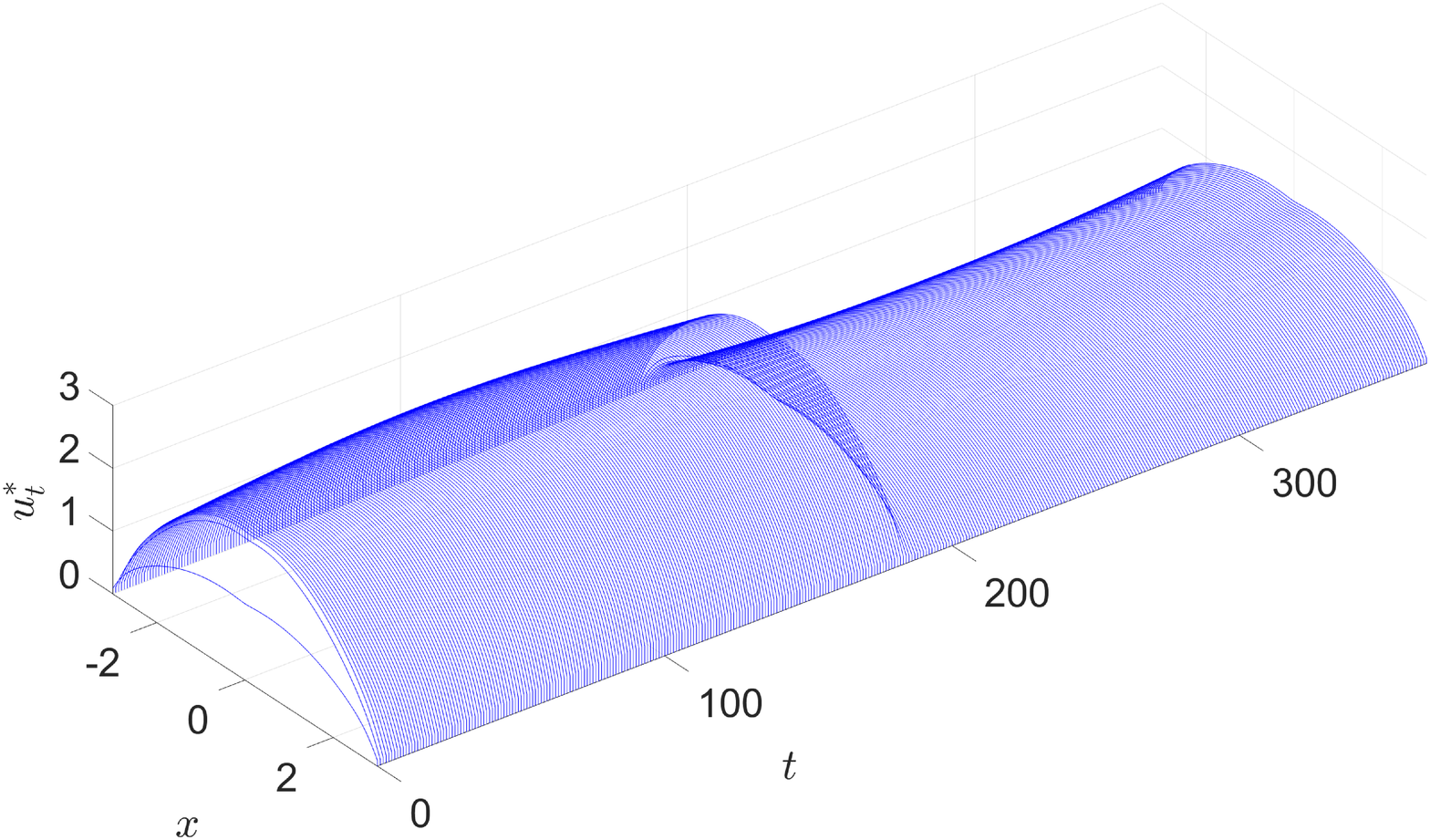}
\includegraphics[trim=125 20 155 105,clip,width=.75\textwidth]{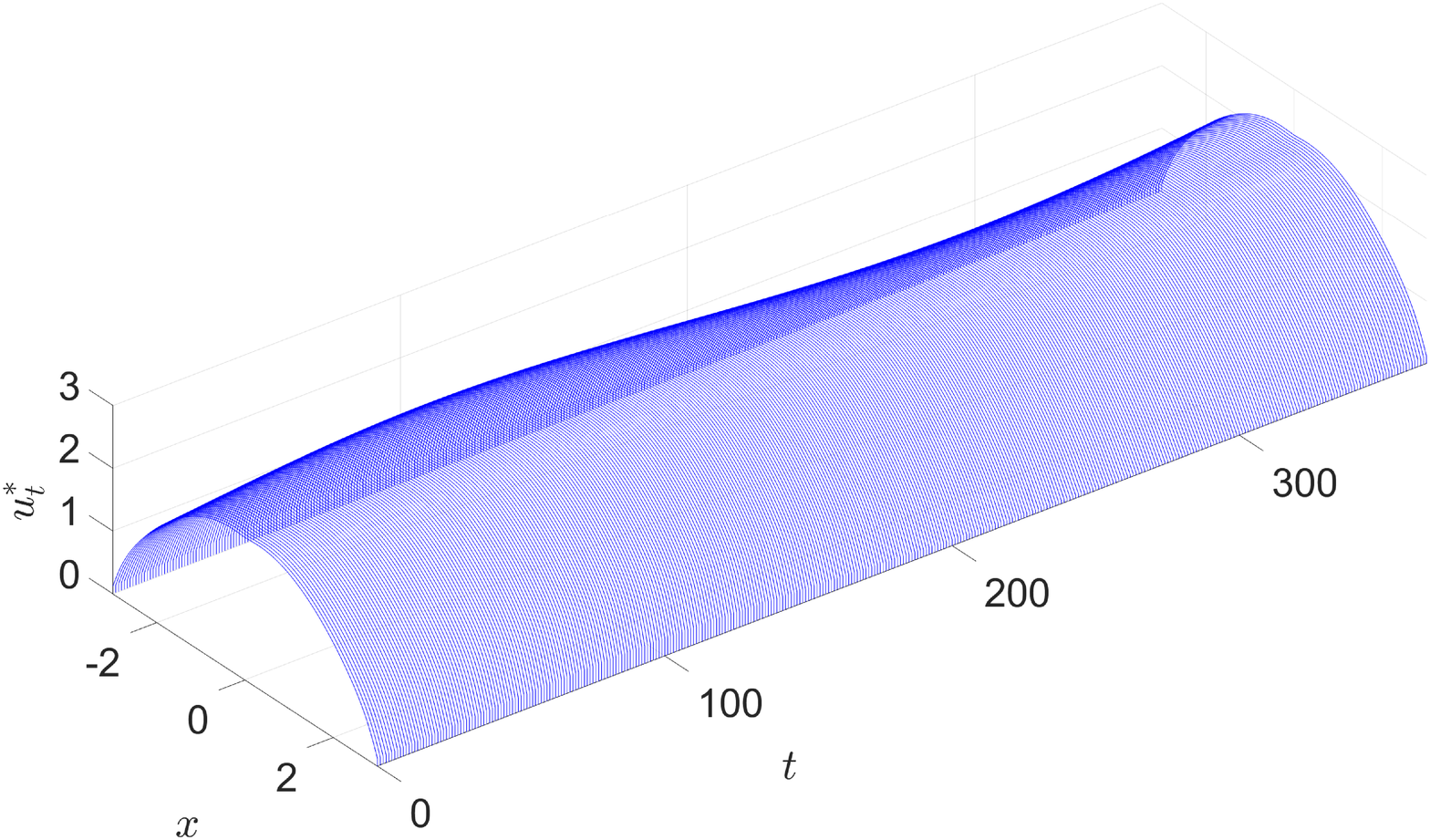}
\caption{The $t$-fibers $\A^\ast(t)$ for $t\in[0,365]$ and $h_t(x)=h^i_t(x)$, $i\in\{1,2,3,4\}$ from the top ($i=1$) to the bottom ($i=4$), respectively.}
\label{figattractor}
\end{center}
\end{figure}
\end{example}

\section*{Acknowledgments}
We are deeply grateful to our supervisor for their support as well as helpful comments and suggestions on the manuscript. We also thank a colleague of ours for polishing the language in some parts.



\end{document}